\documentclass[twoside,12pt,fleqn]{article}
\usepackage{amsmath,indentfirst,amsmath,amsfonts,amssymb,amsthm}
\usepackage{indentfirst,amsmath,amsfonts,amssymb,amsthm,cite}
\usepackage{fancyhdr}
\usepackage[dvips]{graphics}
\usepackage{graphicx}
\usepackage[dvips]{color}
\include{extarrows}
\usepackage{extarrows}

\def\S{\ensuremath {\mathcal S}}

\def\R{\ensuremath {\mathbb R}}

\newcommand{\ra}{\rightarrow}
\newcommand{\fr}{\frac}
\newcommand{\lt}{\leqslant}
\newcommand{\gt}{\geqslant}
\newcommand{\bl}{\biggl}
\newcommand{\br}{\biggr}

\newcommand{\e}{\varepsilon}

\renewcommand{\P}{\ensuremath{\mathbb P}}

\newcommand{\E}{\ensuremath{\mathbb E}}
\newcommand{\F}{\ensuremath{\mathcal F}}

\newcommand{\A}{\ensuremath{\mathcal A}}
\renewcommand{\L}{\ensuremath{\mathcal{L}}}

\newtheorem{rem}{Remark}[section]

\newtheorem{Def}{Definition}[section]
\newtheorem{theo}{Theorem}[section]

\newtheorem{iteration lemma}{iteration Lemma}[section]
\newtheorem{cor}{Corollary}[section]

\begin{document}
%
%
%
%
%

\vspace*{0.4cm}
\begin{center}
\baselineskip=7mm
{\LARGE\bf {Asymptotic properties of a stochastic Gilpin-Ayala model under regime switching}}\\[1cm]
\baselineskip=7mm
 \large\bf
Kai Wang and Yanling Zhu$^{*}$\\[6mm]
 {\footnotesize \it  School of Statistics and Applied Mathematics, Anhui University of Finance and Economics,\\ Bengbu 233030, \ P.R. CHINA}
\begin{figure}[b]
\footnotesize
\rule[-2.truemm]{5cm}{0.1truemm}\\[1mm]
{
\\ $^*$Corresponding author. 
\\
{\it E-mail addresses}\,:
wangkai050318\textit{\char64}163.com(K. Wang);\; zhuyanling99\textit{\char64}126.com(Y.L. Zhu).}
\end{figure}
\end{center}
\vspace{0.03cm}
\noindent\hrulefill \newline

\begin{center}
\begin{minipage}{15cm}\baselineskip=7.mm
\vspace{-3mm} \noindent{\bf Abstract:}\ \ In this paper,  a stochastic Gilpin-Ayala population model with regime switching and white noise is considered. All parameters are influenced by stochastic perturbations. The existence of global positive solution, asymptotic stability in probability, $p$th moment exponential stability, extinction, weak persistence, stochastic permanence and stationary distribution of the model are investigated, which generalize some results in the literatures. Moreover, the conditions presented for  the stochastic permanence and the existence of stationary distribution improve the previous results.


\vspace{1mm}
{\bf AMS(2000):}\ \ 60H10;\ \  60J60;\ \  92D25
\end{minipage}
\end{center}
\vspace{1mm} \noindent\hrulefill \newline
\baselineskip=7.mm

\vspace{2cm}
\section{Introduction}
In order to describe the nonlinear rate change of the population size, Gilpin and Ayala (1973) \cite{ga-73}
proposed a general logistic model (called GA model):
\[d x(t)=x(t)[a-bx^\theta(t)]d t,\]
where $x(t)$  denotes the population size of species at time $t$; $a$ and $b$ are positive constants denote the natural growth rate and death rate of the species;  $\theta>0$ denotes the parameter to modify the classical deterministic logistic model, which is often called GA parameter.
\par  But in the real ecosystem, population systems are always influenced by stochastic environmental noise, which cannot be neglected for all population sizes.
 May (1973)  \cite{may-73}
revealed the fact that due to environmental noise, the birth rate, carrying capacity, competition coefficient and all the other  parameters involved in the system exhibit stochastic fluctuation to a greater or lesser extent. After that stochastic systems become more and more popular, and many authors have done some excellent works in this field. The  pioneer work was due to Khasminskii (1980) \cite{kh-80},
  who established and studied an unstable system by using two white noise sources, his work opened a new chapter in the study of stochastic stabilisation. Mao et. al (2002) \cite{mao-02}
  presented an important claim that the environmental noise can suppress explosions in a finite time in population dynamics.
 \par  Consider the environmental noise in  the birth rate and competition coefficient in GA model, Liu et. al (2012) \cite{liu-12-b}
 and Li (2013) \cite{li-13} presented the following stochastic GA model
 \begin{equation}\label{l2}
d x(t)=x(t)\big[a-bx^{\theta}(t)\big]d t+\sigma_1x(t) dB_1(t)+\sigma_2x^{1+\theta}(t) d B_2(t),
\end{equation}
 where $B_1$ and $B_2$ are two independent 1-dimensional Brownian motions, and studied the stationary distribution, ergodicity and extinction of the model.
 \par It was  known that besides the white noise there is another type of environment noise in the real ecosystem, that is the telegraph noise, which can be demonstrated as a switching between two or more regimes of environment. The regime switching can be modeled by a continuous time Markov chain $(r_t)_{t\ge  0}$ taking values in a finite state spaces $\S=\{1,2,...,m\}$ and with  infinitesimal generator $Q = (q_{ij})\in R^{m\times m}$. That is  $r_t$ satisfies
\[P(r_{t+ \delta  } = j|r_t = i) = \left\{
                                     \begin{array}{ll}
                                     \quad\;\;\, q_{ij}\delta  + o(\delta ), \;\;\text{if}\; i\not= j,\\
                                      1 + q_{ij}\delta  + o(\delta ),  \;\;\text{if}\;  i=j,
                                     \end{array}
                                   \right. \;\text{as}\;\delta\ra 0^+,
\]
 where $q_{ij}\gt0$ is the transition rate from $i$ to $j$ for $i\not=j$,  and $q_{ii}=-\sum_{i\not=j}q_{ij}$ for each $i\in\S,$ see Khasminskii et. al (2007) \cite{kh-07}
 and Zhu et. al (2007, 2009) \cite{zhu-07,zhu-09} for more details. Inspired by this,
 Liu et. al (2011, 2012) \cite{liu-11,liu-12-a}
 considered the following stochastic GA model under regime switching:
 \begin{equation}\label{lm}
d x(t)=x(t)\big[a(r_t)-b(r_t)x^{\theta}(t)\big]d t+\sigma_1(r_t)x(t) dB_1(t)+\sigma_2(r_t)x^{1+\gamma}(t) d B_2(t),
\end{equation}
where $\theta>0$, $\gamma>0$. They studied the existence of global positive solution, persistence, extinction and non-persistence of the species, and obtained the stochastic permanence of the species under the condition that $\gamma\in(0,1]$ and $\theta\in(0,1+\gamma]$. Considered  the claim in May (1973) \cite{may-73} that all parameters involved in ecosystems exhibit stochastic fluctuation, Settati and Lahrouz (2015)  \cite{Sa-15} firstly
presented the GA model with its GA parameter $\theta$ under regime switching:
\begin{equation}\label{sa}
d x(t)=x(t)\big[a(r_t)-b(r_t)x^{\theta(r_t)}(t)\big]d t+\sigma(r_t)x(t) d B(t),
\end{equation}
and investigated the global stability of the trivial solution, and presented sufficient conditions for the extinction,  persistence and existence of stationary distribution of model \eqref{sa}. Under the assumption that $\theta(i)\in(0,1]$ for all $i\in\S$, Liu et. al (2015) \cite{liu-15} investigated the asymptotical stability in probability and the existence of stationary distribution of the following GA model with regime switching
\begin{equation}\label{ll}
d x(t)=x(t)\big[a(r_t)-b(r_t)x^{\theta(r_t)}(t)\big]d t+\sigma_1(r_t)x(t) dB_1(t)+\sigma_2(r_t)x^{1+\theta(r_t)}(t) d B_2(t).
\end{equation}
One can see that model (\ref{ll}) doesn't include the general model (\ref{lm}) which  has important applications in financial field, such as the 3/2 model (or Ahn-Gao model, or Inverse Square Root model)
\[
d x(t)=\mu\, x(t)\big[a-x(t)\big]d t+\sigma x^{3/2}(t) d B(t),
\]
where $x(t)= \rho^2(t)$, and $\rho(t)$ denotes
the instantaneous standard deviation of the stock price returns at time $t$, see \cite{Lewis} for more details.
\par Motivated by above reason, in this paper we consider the telegraph noise in the GA parameters $\theta$ and $\gamma$ in model (\ref{lm}), and get a more general stochastic GA model under regime switching in the following form
\begin{equation}\label{ga}
d x(t)=x(t)\big[a(r_t)-b(r_t)x^{\theta(r_t)}(t)\big]d t+\sigma_1(r_t)x(t) dB_1(t)+\sigma_2(r_t)x^{1+\gamma(r_t)}(t) d B_2(t),
\end{equation}
with initial value $(x_0,r_0)\in\R_+\times\S,$ and for each $i\in\S$, $\theta(i)>0$ and $\gamma(i)>0.$ 
\begin{rem} If $\sigma_2(r_t)\equiv0$, then model \eqref{ga} becomes model \eqref{sa}, and transforms to model \eqref{ll} with $\gamma(i)\equiv\theta(i)$ for each $i\in\S$, and reduces to model \eqref{lm} while there is no switching in the GA parameters $\theta$ and $\gamma.$  If $\sigma_1(r_t)\equiv0$ and $\gamma(r_t)\equiv1/2$, then it is the $3/2$ model. Thus model \eqref{ga} generalizes the previous models.
\end{rem}
The contribution of this paper is that. Compared with the models in the literatures(e.g., \cite{ga-73,li-13,liu-11,liu-12-a,liu-12-b,liu-15,Sa-15}), our model \eqref{ga} provides a more realistic modeling of the population  dynamics, which also includes some important models in financial fields, such as the $3/2$ model, Logistic diffusion model and Double-Well potential model. The results on the existence of global positive solution, asymptotic stability in probability, $p$th moment exponential stability, weak persistence, extinction, stochastic permanence and stationary distribution of the model generalize the  results in previous works. Moreover, the conditions imposed on the positive recurrent, stochastic permanence and the existence of a unique ergodic asymptotically invariant distribution
improve those of Liu et. al (2011, 2012, 2015)(e.g., \cite{liu-11,liu-12-a,liu-15}).

Throughout this paper, we assume that there is a complete
probability space $(\Omega, \F, $ $\{\F_t \}_{t\gt 0}, \P)$ with a filtration $\{\F_t \}_{t\gt 0}$ satisfying
the usual conditions in which the one dimensional Brownian
motions $B_1(t)$ and $B_2(t)$ are defined, and
\vspace{3mm}
\\
\textbf{Assumption 1.} {\it The discrete component $(r_t)_{t\gt0} $ in model \eqref{ga} is an irreducible
continuous-time Markov chain with an invariant distribution $\pi = (\pi_i,i\in\S)$.}
\vspace{2mm}
\\\textbf{Assumption 2.} {\it The Brownian motions $B_1(t), B_2(t)$ and Markov chain $(r_t)_{t\gt0} $ are independent.}
\par
For convenience and simplicity, in this paper we using the following notations:
\[\aligned
\R_+=&(0,\infty);\;\mu(r_t)=a(r_t)-0.5\, \sigma_1^2(r_t) ;\;f_1(x)=\sum_{i\in\S}\pi_i\left[\mu(i)-b(i)x^{\theta(i)}\right];\\
f_2(x)=&\sum_{i\in\S}\pi_i\left[\mu(i)-b(i)x^{\theta(i)}-0.5\,\sigma_2^2(i)x^{2\gamma(i)}\right];
\;\check{a}=\max_{i\in\S}\{a(i)\};\;\hat{a}=\min_{i\in\S}\{a(i)\}.
\endaligned
\]
 \par
 For a function $V: \R_+\times\S\longmapsto\R_+$ such that $V(x,i)$ is twice continuously differential with respect to the first variable $x$ for each $i\in\S,$ we define the operator $\L$ by
\[
\L V(x,i)= x\big[a(i)-b(i)x^{\theta(i)}\big]\fr{\partial V(x,i)}{\partial x}+\fr12 [\sigma_1^2(i)x^2+\sigma_2^2(i)x^{2+2\gamma(i)}]\fr{\partial^2 V(x,i)}{\partial x^2}+\sum_{k\in \S}q_{ik}V(x,k).
\]

\begin{Def}(See Liu et. al (2011)\cite{liu-11} \label{df} for definitions 1-3, Khasminiskii et. al (2007)\cite{kh-07} for definition 4, and Mao et. al (2006) for definitions 5-6)
\par 1. The species $x(t)$ is said to be extinctive if $\lim_{t\ra\infty}x(t)=0;$
\par 2. The species $x(t)$ is said to be weak persistent if $\limsup_{t\ra\infty}x(t)>0;$
\par 3. The species $x(t)$ is said to be stochastically permanent if for any $\e\in(0,1)$, there is a pair of positive constants $\alpha,\beta$ such that $\liminf_{t\ra\infty}\P\{x(t)\gt\beta\}\gt1-\e$ and $\liminf_{t\ra\infty}\P\{x(t)\lt\alpha\}\gt1-\e;$
\par 4. The trivial solution  is said to be asymptotically stable in probability if it is stable in  probability, that is, for any $\e\in(0,1)$ and any $r_0\in\S$, $\lim_{x_0\ra0}\P\{\sup_{t\gt0 } |x_{x_0,r_0}(t)|>\e\}=0$, and satisfying $\lim_{x_0\ra0}\P\{\lim_{t\ra\infty } x_{x_0,r_0}(t)=0\}=1$ for any $r_0\in\S;$
\par 5. For $p>0$, the trivial solution is said to be $p$th moment exponentially stable if for all $(x_0,r_0)\in \R_+\times\S,$ $\lim\sup_{t\ra\infty}\fr1t\log(\E[x^p(t)])<0;$
\par 6.  A square matrix $A=(a_{ij})_{n\times n}$ is called a nonsingular M-matrix if $A$ can be expressed in the form $A=sI-G$ with a nonnegative square matrix $G$(i.e., each element of $G$ is nonnegative) and $s>\rho(G)$, where $I$ is the identity $n\times n$ matrix and $\rho(G)$ the spectral radius of $G.$
\end{Def}
\vspace{3mm}

\par The organization of this paper is as follows. In section 2, the existence of global positive solution of  model \eqref{ga} is proved. In section 3,  the asymptotic stability in probability and $p$th moment exponential stability of the trivial solution to model \eqref{ga} are investigated. In section 4,  some sufficient conditions for the weak persistence and extinction of the species described by model \eqref{ga} are presented. In section 5, the existence of stationary distribution of the solution $x(t)$ to model \eqref{ga} is studied. In section 6,  two examples are given to verify  the theoretical results obtained in previous sections. Some conclusions are given in the last section.

\section{Global Positive Solution}
In this section, we will prove the existence of global positive solution  $x(t)$ to model \eqref{ga} with any initial value $(x_0,r_0)\in\R_+\times\S $. We first prove the existence of unique global solution  $x(t)$ to model \eqref{ga}, and then prove that the solution is almost surely positive for all $t\geqslant0$, that is $\P_{x_0,r_0}(x(t)>0, \forall\,t\gt 0)=1.$

\begin{theo}
For any $(x_0,r_0)\in\R_+\times\S $, there is a unique global solution    $x(t)$ to model \eqref{ga}.
\end{theo}
\begin{proof}
Since all the coefficients of model \eqref{ga} are
locally Lipschitz continuous on $\R_+$,  there is a unique local solution $x(t)$ to model \eqref{ga} with initial value $(x_0,r_0)\in \R_+\times\S$ on
$t\in [0, \tau)$, where $\tau $ is the explosion time.
\par  Now  we show that the solution is globally existent, that is $\tau=\infty$. Let $n_0>0$ be so large that $x_0\in(0,n_0]$. For each $n>n_0$,  define stopping times $\tau_n=\inf\{t\in [0,\tau]\big| \;x(t)\gt n\}$, then $\tau_n$ is increasing as $n\ra\infty.$ Let $\tau_\infty=\lim_{n\ra\infty}\tau_n$, whence $\tau_\infty\lt \tau$ a.s.

We claim that
$\tau_\infty=\infty.$
 Otherwise, there must exist a pair of constants $N>0$ and $\e\in (0,1)$ such that $\P\{\tau_\infty\lt N\}>\e$.
 Therefore, there is an integer $N_1\gt n_0$ such that
$\P\{\tau_n\lt N\}>\e$ for $n\gt N_1.$

Define $C^2$-function $V(x, i)=x^p$ with $p\in(0,1)$, and by the definition of the operator of $\L$ we get
\[
\L V(x,i)=p\,x^p\left[a(i)+\fr12p(p-1)\sigma_1^2(i)-b(i)x^{\theta(i)}\right]+\fr12p(p-1)\sigma_2^2(i)x^{p+2\gamma(i)}\lt K,
\]
for $(x,i)\in \R_+\times \S.$
Then by It$\hat{\text{o}}$'s formula, we have
\[\aligned
V(x(\tau_n\wedge N), r_{\tau_n\wedge N}) = &V(x,i) + \int_{0}^{\tau_n\wedge N} \L V(x_{s},r_{s}) ds \\
&+ p\int_{0}^{\tau_n\wedge N} \sigma_1(r_s)x^p(s)dB_1(s) + p\int_{0}^{\tau_n\wedge N} \sigma_2(r_s)x^{p+\gamma(r_s)}(s)dB_2(s).
\endaligned
 \]
Taking expectation on both sides of the above equality gives
\[
\E[x^p(N\wedge\tau_n)]\lt x_0^p+K\E[N\wedge\tau_n]\lt  x_0^p+KN.
\]
Let $\Omega_n=\{\tau_n\lt N\}$, then $\P(\Omega_n)\gt \e.$ In view of that for every $\omega\in\Omega_n,$ $x(\tau_n,\omega)$ equals to $n$, and we get the following  contradiction:
\[\infty>x_0^p+KN\gt\E[1_{\Omega_n}(\omega)x^p(\tau_n)]\gt\e n^p\ra\infty\;\;as \;\;n\ra\infty.\]
 Thus $\tau_\infty=\infty$.
\end{proof}
\begin{theo}
 For any $(x_0,r_0)\in \R_+\times\S$, the solution $x(t)$ of model \eqref{ga} satisfying $\P_{x_0,r_0}(x(t)>0, \forall\,t\gt 0)=1.$
\end{theo}
\begin{proof}
Define $V(x,i)=x^{-2}$  for $(x,i) \in \R_{+ }\times \S$, and we have
\[\aligned
\L V(x,i)= & x^{-2}\left[2b(i)x^{\theta(i)}+3\sigma_1^2(i)+3\sigma_2^2(i)x^{2\gamma(i)}-2 \right]
:=F(x,i)x^{-2}.
\endaligned
\]
Define for $ 0<\varepsilon<x<K $, $\tau_\e=\inf\{t\gt0|\; x(t)\lt\e\}$, $\tau_K=\inf\{t\gt0|\; x(t)\gt K>1\}$, and $\tau_0=\inf\{t\gt0| \; x(t)=0\}$,
then $\tau_\e\ra\tau_0$ as $\e\ra 0.$ Meanwhile, it is well known from Theorem 2.1 that the solution $x(t)$ has no finite explosion time and hence $\tau_K\ra\infty$ a.s. as $K\ra \infty.$
\par  Set $D(t)=t\wedge\tau_\e\wedge\tau_K$,
then by It$\hat{\text{o}}$'s formula, we obtain
\[\aligned
\E\big[x^{-2}(D(t))\big]&=x_0^{-2}+\E\int_0^{D(t)}\L V(x(s),r_s)ds\lt x_0^{-2}+\check{F}\,\E\int_0^{D(t)} x^{-2}(s)ds,
\endaligned
\]
where $\check{F}=\max_{i\in\S} F(K,i).$
It follows from Gronwall's inequality that
\[
\E\big[x^{-2}(D(t))\big]\lt x_0^{-2}e^{\check{F}\,t}.
\]
If $\P(\tau_0<\infty)>0$, then we can choose $t$ and $K$ large enough such  that $\P(\tau_0<t\wedge\tau_K)>0.$  So by Chebeshev's inequality, we get
\[\aligned
0<&\,\P(\tau_0<t\wedge\tau_K)\lt \P(\tau_\e<t\wedge\tau_K)\lt \P(x(D(t))\lt\e)\\
=&\,\P(x^{-2}(D(t))\gt\e^{-2})\lt\e^{2}\E[x^{-2}(D(t))]
\lt\, \e^{2} x_0^{-2}e^{\check{F}\,t}\ra 0\;\;as \;\e\ra0,
\endaligned
\]
which is a contradiction, then $\P(\tau_0<\infty)=0$ is obtained.
\end{proof}
From above two theorems, one can easily get the following result.
\begin{cor} For any $(x_0,r_0)\in \R_+\times \S$, the solution $x(t)$ to model \eqref{ga} is global existence and remains in $\R_+$ for $t\gt0.$
\end{cor}
\section{Stability of Trivial Solution}
\begin{theo} \label{tth} For any $(x_0,r_0)\in\R_+\times\S.$ If $\sum_{i\in\S}\pi_i\mu(i)<0,$ then the trivial solution to model \eqref{ga} is asymptotically stable in probability.
\end{theo}
\begin{proof}
Let $\zeta=(\zeta_1,...,\zeta_m)^T$ be a solution of the Poisson system:
\begin{equation}\label{ps}
Q\zeta=-\mu+\sum_{i\in\S}\pi_i\mu(i)\,\mathbf{1},
\end{equation}
where $\mu=(\mu(1),\mu(2),\cdots,\mu(m))^T$. Choose sufficient large positive constant $p$ such that
\[
\fr{p}{p+\zeta_i}\sum_{i\in\S}\pi_i\mu(i)
+\fr{\sigma_1^2(i)}{2p}+\fr{\zeta_i\mu(i)}{p+\zeta_i}<0,\;\;p>1 \;\;\mbox{and}\;\; p+\zeta_i>0 \;\;\text{for all}\;\;i\in\S.
\]
Define $C^2$-function $V(x,i)=(p+\zeta_i)x^{1/p},\;i\in\S,$ we obtain
\[\aligned
&\L V(x,i)\\
=&\fr1pV(x,i)\bl[\mu(i)+\sum_{k\in\S}q_{ik}\zeta_k-b(i)x^{\theta(i)}+\fr{1-p}{2p}\sigma^2_2(i)x^{2\gamma(i)}+\fr{\sigma_1^2(i)}{2p}-\fr{\zeta_i}{p+\zeta_i}\sum_{k\in\S}q_{ik}\zeta_k\br]\\
=&\fr1pV(x,i)\bl[\sum_{i\in\S}\pi_i\mu(i)-b(i)x^{\theta(i)}-\fr{p-1}{2p}\sigma^2_2(i)x^{2\gamma(i)}+\fr{\sigma_1^2(i)}{2p}+\fr{\zeta_i\mu(i)}{p+\zeta_i}-\fr{\zeta_i}{p+\zeta_i}\sum_{i\in\S}\pi_i\mu(i)\br]\\
\lt &-Cx^{1/p},
\endaligned\]
where $C=C(p)$ is a positive constant. Then, for any  sufficient small $\e\in(0,r)$ we have
\[
\L V(x,i)\lt -C\e^{1/p}\;\;\mbox{for any}\;x\in(\e,r)\;\mbox{and}\; i\in\S.
\]
 Thus according to Lemma 3.3 and Remark 3.5--(i) in Khasminskii et. al (2007) \cite{kh-07}, we get that the trivial solution of  model \eqref{ga}
  is  asymptotically stable in probability.
  \end{proof}

  \begin{theo}\label{thr} For $p\in(0,1)$, if $\A(p):=\text{diag}(h_1(p),\cdots,h_m(p))-Q$ is a nonsingular M-matrix, then the trivial solution of  model \eqref{ga} is $p$th moment exponentially stable, where
  $h_i(p)=\frac12p(1-p)\sigma_1^2(i)-p\, a(i)$, $i\in\S$.
  \end{theo}

  \begin{rem} If $h_i(p)>0$ for all $i\in\S$, then all the row sums of $\A(p)$ are positive, thus according to Minkovski Lemma, $\det(\A (p))$ is positive. Furthermore, by the  properties of generator $Q$  we know that all the principle minors of $\A(p)$ are positive, thus according to Theorem 2.10-(2) of Mao et. al (2006) \cite{mao-06} that $\A(p)$ is a nonsingular M-matrix. Then it follows from Theorem \ref{thr} that the trivial solution of  model \eqref{ga} is $p$th moment exponentially stable, which also implies the asymptotic stability in probability.
   \par On the other hand, $h_i(p)>0$ for all $i\in\S$ imply $\mu(i)<0$ for all $i\in\S$, thus according to Theorem \ref{tth}, the trivial solution  of  model \eqref{ga} is asymptotic stability in probability.
\end{rem}

  \begin{proof} It follows from Theorem 2.10-(9) of Mao et. al (2006) \cite{mao-06} that there is a vector $\beta=(\beta_1,\cdots,\beta_m)^T>0$, i.e., $\beta_i>0$ for all $1\lt i\lt m$, such that
  \[(\overline \beta_1,\cdots,\overline\beta_m)^T:=\A (p)\beta>0.\]
  Then
  \[h_i(p)\beta_i-\sum_{j\in\S} q_{ij}\beta_j=\overline \beta_i>0\;\;\text{for}\; 1\lt i\lt m.\]
  Define the Lyapunov in the form $V(x,i)=\beta_ix^p$, then we obtain
\[
\aligned
\L V(x,i)=&\bl[p\beta_i (a(i)-b(i) x^{\theta(i)})+\frac12p(p-1)\beta_i \sigma_1^2(i)+\frac12p(p-1)\beta_i \sigma_2^2(i)x^{2\gamma(i)} +\sum_{k\in\S}q_{ik}\beta_k\br]x^p\\
\leq& \bl[\big(p\, a(i)+\frac12p(p-1)\sigma_1^2(i)\big)\beta_i +\sum_{k\in\S}q_{ik}\beta_k\br]x^p\\
=&-\bl[h_i(p)\beta_i -\sum_{k\in\S}q_{ik}\beta_k\br]x^p\lt -\lambda x^p,
\endaligned
\]
where $\lambda=\min_{1\lt i\lt m}\overline\beta_i.$ Thus according to Theorem 5.8 in Mao et. al (2006) \cite{mao-06} we get the result.
\end{proof}

\section{Weak Persistence and Extinction}
\begin{theo} \label{t31} If  $\sum_{i\in\S}\pi_i\mu(i)>0$, then
\[\liminf_{t\ra\infty}x(t)\lt x_*\lt x^*\lt\limsup_{t\ra\infty}x(t),\]
where $x^*$ and $x_*$ are the unique positive solutions of equations $f_1(x)=0$ and $f_2(x)=0$, respectively.
\end{theo}
\begin{rem}  $x^*\gt x_*>0$ can be obtained from the fact that $f_1\gt f_2$ and $f_2(0)>0$.
Theorem \ref{t31} shows that solutions of  model \eqref{ga} will oscillate infinitely often from $x_*$ up to $ x^*,$ and the amplitude is no less than
 $A=x^*-x_*$, which may be decreasing with the decreasing of $\sigma_2^2.$ If $\sigma_2\not=0,$ then $A>0$, and then  $\liminf_{t\ra\infty}x(t)\lt x_*< x^*\lt\limsup_{t\ra\infty}x(t).$ Moreover, this theorem implies the weak persistence of the species because of  $\limsup_{t\ra\infty}x(t)\gt x^*>0.$
\end{rem}
\begin{rem}
If $\sigma_2(r_t)\equiv0$, then $f_1(x)=f_2(x)$, $x^*= x_*$ and the equalities hold, which is the case of Theorem 3.1 in Settati et. al (2015) \cite{Sa-15}. Thus our result is an extension of it.
\end{rem}
\begin{proof}
One can see that $f_j$, $j=1,2$ are continuous and strictly decreasing on $\R_+$ and
\[
f_j(0^+)=\sum_{i\in\S}\pi_i\mu(i)>0,\;\;f_j(\infty)=-\infty.
\]
So there exist unique positive solutions $x^*$, $x_*$($x_*<x^*$), such that $f_1(x^*)=0$, $f_2(x_*)=0$, respectively.
\par \textbf{Step 1.} Assume $\P(\omega\in\Omega,\;\limsup_{t\ra\infty} x(t,\omega)<x^*)>0,$ then there exists a positive constant $\alpha\in(1/2,1)$ such that
$\P(\Omega_1)>0,$ where $\Omega_1=\{\omega\in\Omega,\;\limsup_{t\ra\infty} x(t,\omega)<(2\alpha-1)x^*\}.$ So for every $\omega\in\Omega_1,$ there is a $T(\omega)>0$ such that
\[
x(t)\lt (2\alpha-1)x^*+(1-\alpha)x^*=\alpha x^*\;\;
 \mbox{for all}\;\; t\gt T(\omega).\]
Then it follows from  model \eqref{ga}  that
\begin{equation}\aligned\label{wl}
\log x(t)
\gt& \log x_0+\int_0^t\mu(r_s)dt-\int_0^T\left[b(r_s)x^{\theta(r_s)}+\fr12\sigma_2^2(r_s)x^{2\gamma(r_s)}\right]dt\\
&-\alpha^{\hat\theta\wedge2\hat\gamma}\int_T^t\left[b(r_s)( x^*)^{\theta(r_s)}+\fr12\sigma_2^2(r_s)(x^*)^{2\gamma(r_s)}\right]dt+M_1(t)+M_2(t),
\endaligned
\end{equation}
where $M_1(t)=\int_0^t\sigma_1(r_s)dB_1(t),\;M_2(t)=\int_0^t\sigma_2(r_s)x^{\gamma(r_s)} dB_2(t).$
Notice that $M_i(t)$ are real valued continuous local martingale with the quadratic variations:
\[
\langle M_1(t), M_1(t)\rangle=\int_0^t\sigma_1^2(r_s)ds\lt \check {\sigma_1^2}t,
\]
and
\[
\langle M_2(t), M_2(t)\rangle\lt\alpha^{2\check\gamma} \big[(x^*)^{2\hat\gamma}+(x^*)^{2\check\gamma}\big]\int_0^t\sigma_2^2(r_s)ds\lt \check {\sigma_2^2}\alpha^{2\check\gamma} \big[(x^*)^{2\hat\gamma}+(x^*)^{2\check\gamma}\big]t.
\]
Thus by the large number theorem for martingales and the ergodic theory of the Markov chain, we obtain from inequality \eqref{wl} that there is $\Omega_1'\subset\Omega$ such that $\P(\Omega_1')=1$, and for every $\omega\in\Omega_1'$,
\[
\aligned
\liminf_{r\ra \infty}\fr1t\log x(t)&\gt \sum_{i\in \S}\pi_i\left\{\mu(i)-\alpha^{\hat\theta\wedge2\hat\gamma}\biggl[b(i)(x^*)^{\theta(i)}+\fr{\sigma_2^2(i)}2(x^*)^{2\gamma(i)}\biggr]\right\}\\
&\gt (1-\alpha^{\hat\theta\wedge2\hat\gamma})\sum_{i\in \S}\pi_i\mu(i)>0,
\endaligned
\]
which implies $\lim_{t\ra\infty} x(t)=\infty$. But this is a contradiction. Thus
\[
x^*\lt\limsup_{t\ra\infty} x(t).
\]
\par \textbf{Step 2.} Assume $\P(\omega\in\Omega,\;\liminf_{t\ra\infty} x(t,\omega)>x_*)>0$, then  there exists a constant $\beta>1$ such that $\P(\Omega_2)>0$, where $\Omega_2=\{\omega\in \Omega,\;\liminf_{t\ra\infty} x(t,\omega)\gt (2\beta-1)x_*\}$.
Thus, for every $\omega\in\Omega_2$, there exists a $T_1(\omega)>0$ such that
\[x(t)\gt (2\beta-1)x_*-(\beta-1)x_*=\beta x_*\;\;\mbox{for}\;\; t\gt T_1(\omega).\]
Then we have
\begin{equation}\label{logeq}
\aligned
\log x(t)=&\log x_0+\int_0^t\mu(r_s)dt-\int_0^{t}\left[b(r_s)x^{\theta(r_s)}+\fr12\sigma_2^2(r_s)x^{2\gamma(r_s)}\right]dt+M_1(t)+M_2(t).
\endaligned
\end{equation}
It follows from the exponential martingale inequality, see Applebaum (2009)\cite{ad-09},
that for any positive numbers $T$, $\e$ and $\delta$,
\[
\P\{\sup_{t\in[0,T]}[M_2(t)-\fr\e2\langle M_2(t), M_2(t)\rangle]>\delta\}\lt e^{-\e\delta}.
\]
Let $T=T_1$, $\e=1$ and $\delta=2\log T_1$, then we have
\[
\P\{\sup_{t\in[0,T_1]}[M_2(t)-\fr\e2\langle M_2(t), M_2(t)\rangle]>2\log T_1\}\lt \fr1{T_1^2}.
\]
By the Borel-Cantelli's Lemma,  for almost all $\omega\in\Omega$, there is a random integer $n_0=n_0(\omega)$ such that
\[
\sup_{t\in [0,T_1]}\big[M_2(t)-\fr12\langle M_2(t), M_2(t)\rangle\big]\lt 2\log T_1\;\;\mbox{for all}\;\; T_1\gt n_0\;\;\mbox{a.s.},
\]
which yields
\[
M_2(t)\lt 2\log T_1+\fr12 \int_0^t\sigma_2^2(r_s)x^{2\gamma(r_s)}ds\;\;\mbox{for all}\;\; t\in[0,T_1]\;\;\mbox{with}\;\; T_1\gt n_0\;\;\mbox{a.s.}
\]
Inserting the  inequality into \eqref{logeq} gives
\[
\aligned
\log x(t)\lt &\log x_0+\int_0^t\mu(r_s)dt-\beta^{\hat\theta}\int_{T_1}^{t}b(r_s)(x_*)^{\theta(r_s)}dt-\int_0^{T_1}b(r_s)x^{\theta(r_s)}dt+M_1(t) +2\log T_1.
\endaligned
\]
Similarly, by using the large number theorem for martingale $M_1(t)$ and the ergodic theory of the Markov chain, we obtain
\[\aligned
\limsup_{t\ra\infty}\fr1t
\log x(t)&\lt \sum_{i\in\S}\pi_i\big[\mu(i)-\beta^{\hat\theta}b(i)(x_*)^{\theta(i)}\big]=\big(1-\beta^{\hat\theta}\big)\sum_{i\in\S}\pi_i\mu(i)<0,
\endaligned
\]
which leads to a contradiction.  The proof is now completed.
\end{proof}
\begin{cor}\label{c4.1}
For any $(x_0,r_0)\in \R_+\times\S$, the solution $x(t)$ of  model \eqref{ga} has the following property:
\[\limsup_{t\ra\infty}\fr1t\log x(t)\lt\sum_{i\in\S}\pi_i\mu(i)\;\;a.s.\]
Moreover, if $\sum_{i\in\S}\pi_i\mu(i)<0,$ then the species $x(t)$ will be extinctive a.s.
\end{cor}
\section{Stationary Distribution}
\begin{theo}\label{th} If $\theta(i)+1\gt2\gamma(i)$ for each $i\in\S$ and $\sum_{i\in\S}\pi_i\mu(i)>0,$ then, for any $(x_0,r_0)\in\R_+\times\S,$ the  solution $x(t)$ of model \eqref{ga} is positive recurrent and admits a unique ergodic  asymptotically invariant distribution $\nu$. Moreover, the species described by model \eqref{ga} is stochastically permanent.
\end{theo}
\begin{rem} One can see that  our condition $\theta(i)+1\gt2\gamma(i)$ for each $i\in \S$  holds for large $\theta(i)$,  and $\gamma(i)$ is allowed to be bigger than $1$ while $\theta(i)$ is larger than $1$. Meanwhile, if for $i\in\S$, $\gamma(i)=\theta(i)$ and $\theta(i)\in(0,1]$(which is assumed to guarantee the positive recurrent and the existence of a unique ergodic asymptotically invariant distribution in Liu et. al (2015)\cite{liu-15}), then $\theta(i)+1\gt2\gamma(i)$ holds.  For the case of  no switching in the GA parameters $\theta$ and $\gamma$, under the conditions that $\gamma\in(0,1]$ and $\theta\in (0,1+\gamma]$, the stochastic  permanence  is obtained  in   Liu et. al (2011, 2012) \cite{liu-11,liu-12-a}. Therefore, our conditions improve the corresponding ones in these literatures.
\end{rem}

\begin{proof}
Define a Lyapunov function in the form:
\[
V(x,i)=(1-p\zeta_i)x^{-p}+x,
\]
where $p$ is positive number satisfying $1>p\max_{i\in \S}\{\zeta_i\}$,  
and $\zeta=(\zeta_1,...,\zeta_m)^T$ is a solution of the
Poisson system \eqref{ps}. Then  we have
\begin{equation}\label{lv}
\aligned
\L V(x,i)=&-p(1-p\zeta_i)x^{-p}\big[a(i)-b(i)x^{\theta(i)}\big]+\fr12p(p+1)(1-p\zeta_i)\sigma_1^2(i) x^{-p}\\
&+\fr12p(p+1)(1-p\zeta_i)\sigma_2^2(i) x^{2\gamma(i)-p}-px^{-p}\sum_{k\in \S}q_{ik}\zeta_k+x\big[a(i)-b(i)x^{\theta(i)}\big]\\
=&-p(1-p\zeta_i)x^{-p}\bl[a(i)-\fr12(p+1)\sigma_1^2(i)+\fr{1}{1-p\zeta_i} \sum_{k\in \S}q_{ik}\zeta_k\br]\\
&+x\big[a(i)-b(i)x^{\theta(i)}\big]+p(1-p\zeta_i)b(i)x^{\theta(i)-p}+\fr12p(p+1)(1-p\zeta_i)\sigma_2^2(i) x^{2\gamma(i)-p}\\
=&-p(1-p\zeta_i)x^{-p}\bl[\sum_{i\in\S}\pi_i\mu(i)-\fr12p\,\sigma_1^2(i)+\fr{ p\zeta_i}{1-p\zeta_i}\sum_{k\in \S}q_{ik}\zeta_k \br]+x\big[a(i)-b(i)x^{\theta(i)}\big]\\
&+p(1-p\zeta_i)b(i)x^{\theta(i)-p}+\fr12p(p+1)(1-p\zeta_i)\sigma_2^2(i) x^{2\gamma(i)-p}.
\endaligned
\end{equation}
Let $U(N)=(1/N,N)\subset \mathbb{R}_+$. In view of $\sum_{i\in\S}\pi_i\mu(i)>0$,  we can choose $p$ sufficient small such that
\[
\sum_{i\in\S}\pi_i\mu(i)-\fr12p\,\sigma_1^2(i)+\fr{ p\zeta_i}{1-p\zeta_i}\sum_{k\in \S}q_{ik}\zeta_k>0,
\]
and note that $\theta(i)+1\gt 2\gamma(i)$ for each $i\in \S$, then we have
\[\L V(N,i)\ra -\infty$, as $N\ra \infty.\] Thus, for any given positive constant $K$, there exists a sufficient large $N_0$ such that
\[\L V(x,i)\lt -K\;\mbox{ for all }\; x\in U^c(N_0).\] Then according to the Theorems 3.13 and 4.3 in Zhu et. al (2007) \cite{zhu-07}
or Theorem 4.1 in Settati et. al (2015) \cite{Sa-15}
we obtain the first part of the theorem.
\par Now we prove the second part. By the ergodicity of $x(t)$, we have
\[
\fr1t\int_0^tI_{\{x(s)\in \R_+\}}ds\ra\int_0^\infty I_{\{\in \R_+\}}(x)\pi(dx)=\nu(\R_+),
\]
which together with $x(t)\in \R_+$ yields $\nu(\R_+)=1.$ It follows from the asymptotically invariant distribution of $x(t)$ that, for positive constants $\alpha$ and $\beta$,
\[\liminf_{t\ra\infty}\P(x(t)\gt\alpha)=\nu\big([\alpha,\infty)\big)\;\;\text{and}\;\;\liminf_{t\ra\infty}\P(x(t)\lt\beta)=\nu\big((0,\beta)\big).\]
Thus
\[\lim_{\alpha\ra0^+}\nu([\alpha,\infty))=\lim_{\beta\ra\infty}\nu((0,\beta))=\nu(\R_+)=1.\]
Then  for any $\e\in (0,1) $ there exists a sufficient small positive constant $k$ such that
\[
\liminf_{t\ra\infty}\P(x(t)\gt k)\gt1-\e$ and $\liminf_{t\ra\infty}\P(x(t)\lt 1/k)\gt1-\e.
\]
This proof is now completed.
\end{proof}
\begin{rem} From Eq.\eqref{lv} in the  proof we find
that the condition $\theta(i)+1\gt2\gamma(i)$ for each $i\in\S$ can be removed from Theorem \ref{th} in the following three cases:
\par  i) \  For all $i\in\S$, $\sigma_2(i)=0$,  then model \eqref{ga} reduces to model \eqref{sa}, and the condition $\theta(i)+1\gt2\gamma(i)$ for each $i\in\S$ can be removed,  then we get the same result as that in Settati et. al (2015) \cite{Sa-15}.
\par  ii)\   For all $i\in\S$, $\sigma_1(i)=0$, and there is no switching in the GA parameters, that is for all $i\in\S$, $\theta(i)=\theta=const.$, and $\gamma(i)=\gamma=const.$, then $\mu(i)=a(i)$. In this case we can choose sufficient large $p$ such that $\theta+1>2\gamma-p>0,$ and thus the condition $\theta(i)+1\gt2\gamma(i)$ for each $i\in\S$ can also be removed;
\par iii) \ For all $i\in\S$, $\gamma(i)=\theta(i)$ and $\theta(i)\in (0,1],$ which are presented in Liu et. al (2005)\cite{liu-15}. In this case, the condition $\theta(i)+1\gt2\gamma(i)$ holds for each $i\in\S$.
\end{rem}

\section {Examples}
In order to verify the theoretical results obtained in previous sections, we give the following two examples. The numerical method used here is Milstein’s Higher Order Method, see Higham (2001) \cite{Higham-01} for more details.
\par Set the states space of Markov chain $r_t$ by $\S=\{1,\,2,\,3,\,4\}$, and its  generator $Q$ by
\[Q=\left(
      \begin{array}{cccc}
      -10 & 3 & 2 & 5 \\
        6 & -9 & 2 & 1\\
        3 & 3 & -8 & 2\\
        1 & 5 & 3 & -9\\
      \end{array}
    \right)\]
   Then its one step transition probability matrix $P$ is as follows,
    \[
P=\exp(\Delta \cdot Q)=\left(
      \begin{array}{cccc}
       0.9990 &   0.0003 &   0.0002 &   0.0005\\
       0.0006 &   0.9991 &   0.0002 &   0.0001\\
       0.0003 &   0.0003 &   0.9992 &   0.0002\\
       0.0001 &   0.0005 &   0.0003 &   0.9991\\
      \end{array}
    \right),
 \]
 where $\Delta=1e-4$ is the step, and its stationary distribution is
\[\pi=( 0.2622,\,    0.2879,\,    0.2227,\,    0.2272).\]
The computer simulation of the Markov chain is shown in Figure 1.

\begin{center}
{\includegraphics[scale=0.8]{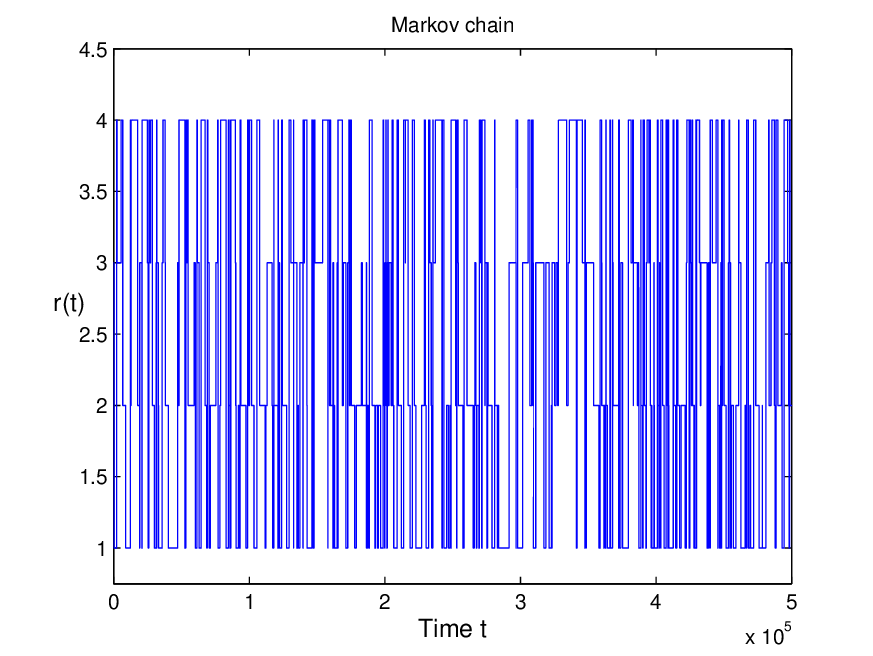}}\\
\footnotesize {\textbf{Figure 1.}\;\;
The simulation of Markov chain $(r_t)_{t\geqslant0}$ with initial state $r_0=3$ and step $\Delta=1e-4$.}
\end{center}

\vspace{2mm}
\par \textbf{Example 1.} Let
\[
\aligned
&a=(0.4,\,0.3,\,0.6,\,0.55),\;b=(0.15,\,0.2,\,0.13,\,0.4),\;\sigma_1=(0.3,\,0.2,\,1.4,\, 0.5),\\
&\sigma_2=(0.13,\,0.21,\,0.11,\,0.24),\,\;\theta=(1.5,\, 0.5,\, 1,\, 0.7),\,\; \gamma=(1.2,\, 0.6,\, 1,\, 0.8).
\endaligned
\]
 Then $\mu=(0.3550,\,    0.2800,\,   -0.3800,\,   0.4250)$ and $\sum_{i\in\S}\pi_i\mu(i)= 0.1856>0,$ and $x_*=0.7681,\;x^*=0.8280$.
 \par Thus according to Theorems \ref{t31}, \ref{th} and Definition \ref{df}, we obtain the weak persistence and the stochastic permanence of the species $x$ of model (\ref{ga}). The evolution of $x(t)$ along with the Markov chain $(r_t)_{t\geqslant0}$ is simulated in Figure 2, and the evolutions of $x(t)$ with $r_t\equiv i$, $i\in \S$ are shown in Figure 3.  The density and distribution of $x(t)$ along with the Markov chain $(r_t)_{t\geqslant0}$ are shown in Figure 4.

 \begin{center}
\scalebox{0.8}{\includegraphics{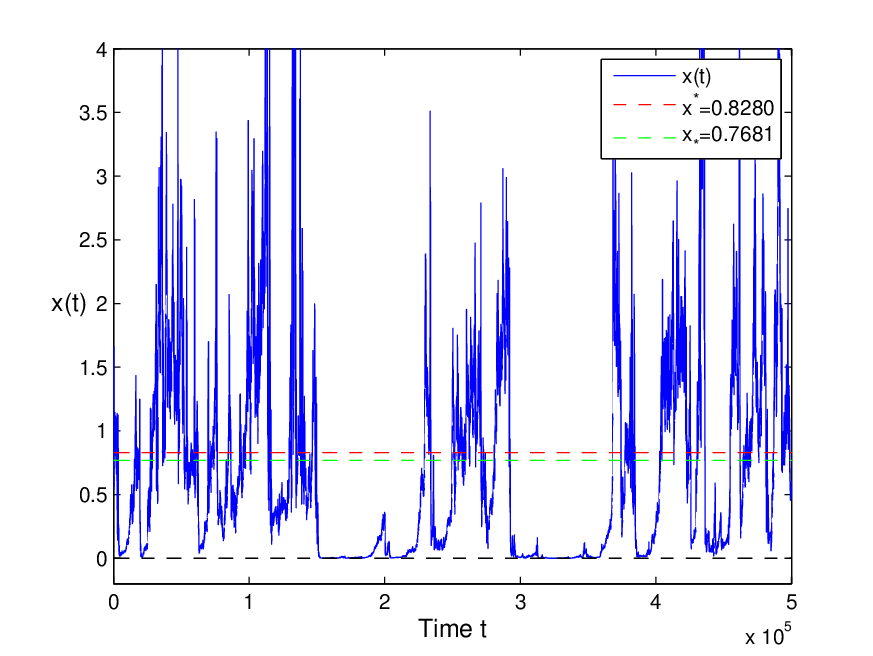}}\\
\footnotesize {\textbf{Figure 2.}\;\;
The evolution  of $x(t)$ along with Markov chain $(r_t)_{t\geqslant0}$ with initial value $x(0)=1$ and step $\Delta=1e-3$, and the unique positive solutions of $f_i(x)=0(i=1,2)$ are given by $x_*$ and $x^*$, respectively.}
\end{center}

\begin{center}
\scalebox{0.8}{\includegraphics{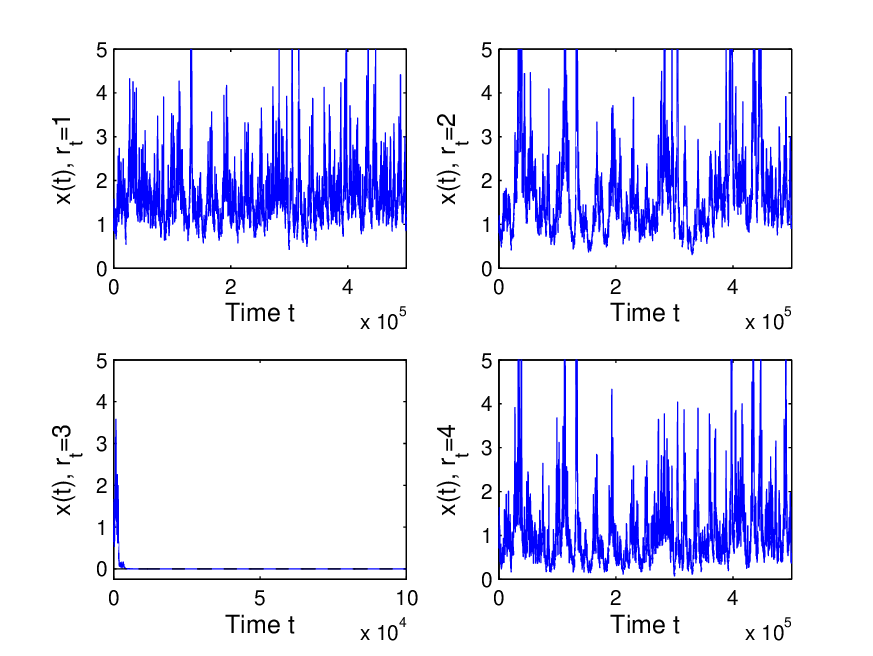}}\\
\footnotesize {\textbf{Figure 3.}\;\;
The evolution  of $x(t)$ for $r_t\equiv i(i\in\S),$ respectively, with initial value $x(0)=1$ and step $\Delta=1e-3$.
These figures show that $x(t)$ with $r_t\equiv3$ is stable and vanish finally(left down figure), but the other states are unstable.}
\end{center}
\begin{center}
\scalebox{0.8}{\includegraphics{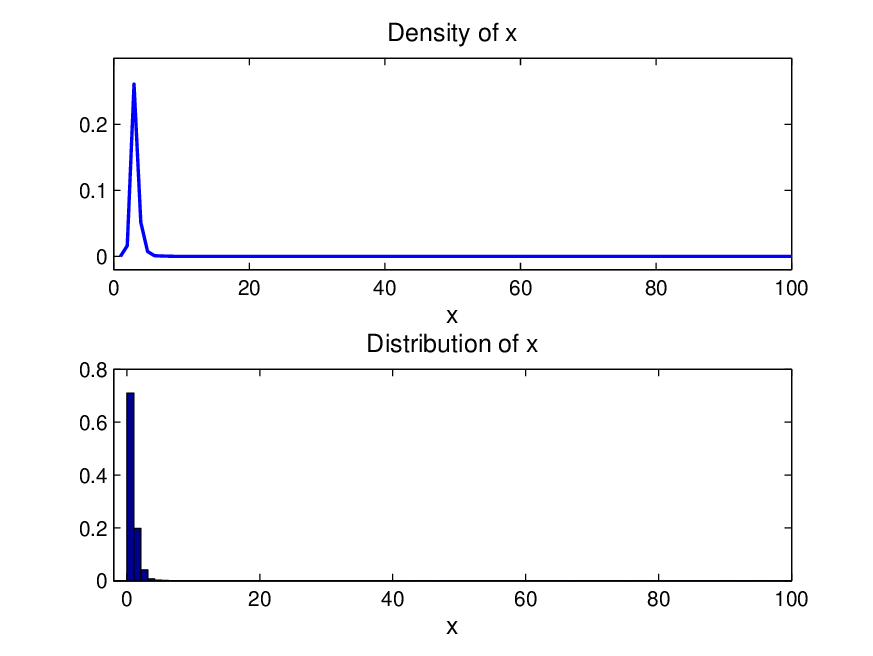}}\\
\footnotesize {\textbf{Figure 4.}\;\;
The density and distribution of $x(t)$ along with Markov chain $(r_t)_{t\geqslant0}$, where $x(t)$ with initial value $x(0)=1$ and  step $\Delta=1e-3$.}
\end{center}

\vspace{2mm}
\par From Figure 3, we find that the solution $x(t)$ of model (\ref{ga}) with $r_t\equiv3$ goes to zero, that is the trivial solution of model (\ref{ga}) with $r_t=3$ is stable, but the other three states of it are unstable. However, the solution $x(t)$ along with the Markov chain $(r_t)_{t\geqslant0}$, which shown in Figure 2 is unstable finally because of $\sum_{i\in\S}\pi_i\mu(i)>0$.

\vspace{2mm}
\par \textbf{Example 2.} Replace $\sigma_1(4)=0.5$ by $\sigma_1(4)=1.8$ in Example 1. Then we have \[\mu=(0.3550,\,    0.2800,\,   -0.3800,\,    -1.0700)\;\;\mbox{and}\;\; \sum_{i\in\S}\pi_i\mu(i)=-0.1540<0.\]
Hence, it follows from  Theorem \ref{tth} and Corollary \ref{c4.1} that the trivial solution to model (\ref{ga}) is asymptotically stable in probability, and the species $x$ will go to be extinctive finally. The computer simulation results of this example are given by Figures 5-6.
\begin{center}
\scalebox{0.8}{\includegraphics{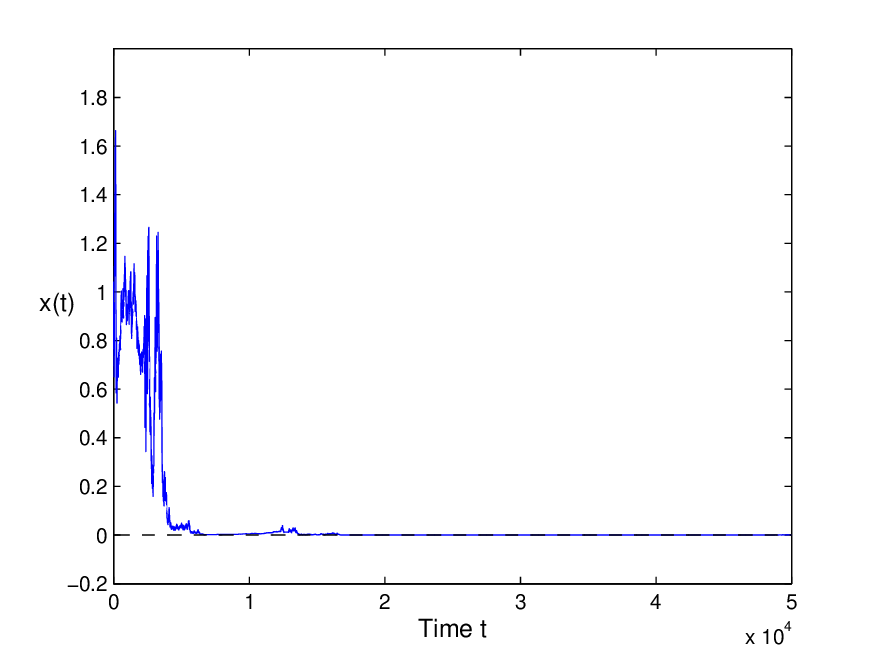}}\\
\footnotesize {\textbf{Figure 5.}\;\;
The evolution of $x(t)$ along with Markov chain $r_t$, where $x(t)$ with initial value $x(0)=1$ and  step $\Delta=1e-3$.}
\end{center}

\begin{center}
\scalebox{0.8}{\includegraphics{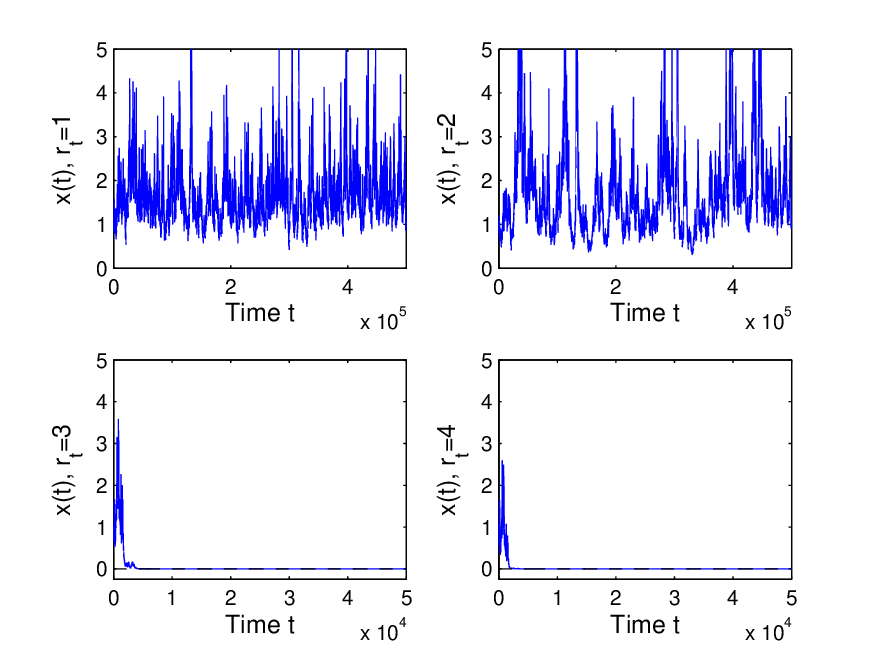}}\\
\footnotesize {\textbf{Figure 6.}\;\;
The evolution  of $x(t)$ for $r_t\equiv i, i\in\S,$ respectively, with initial value $x(0)=1$ and step $\Delta=1e-3$.
These figures show that $x(t)$ with $r_t\equiv3$, $r_t\equiv4$ are stable and vanish finally(down figures), but the other two states are unstable.}
\end{center}

\par Moreover, let $p=0.5$, we have $h(p)=(-0.1888,\,   -0.1450,\,   -0.0550,\,    0.1300)$ and
\[
\A(p)=\left(
        \begin{array}{cccc}
          9.8112  & -3.0000 &  -2.0000 &  -5.0000\\
         -6.0000  &  8.8550 &  -2.0000 &  -1.0000\\
         -3.0000  & -3.0000 &   7.9450 &  -2.0000\\
         -1.0000  & -5.0000 &  -3.0000 &   9.1300\\
        \end{array}
      \right)_.
\]
Its eigenvalues are \[\lambda=( -0.0751,\,12.7683 + 3.0485i,\,12.7683 - 3.0485i,\,10.2797),\] which yield $\A(p)$ is not a nonsingular $M$-matrix. Thus Theorem {\ref{thr}} is invalid, but the computer simulation shows that the trivial solution of model (\ref{ga}) is $0.5$th moment exponentially stable, see Figure 7 for more details. This shows that the assumption, $\A(p)$ is a nonsingular $M$-matrix, is not a necessary condition for the $p$th moment exponential stability of $x$ to model (\ref{ga}).

\begin{center}
\scalebox{0.8}
{\includegraphics{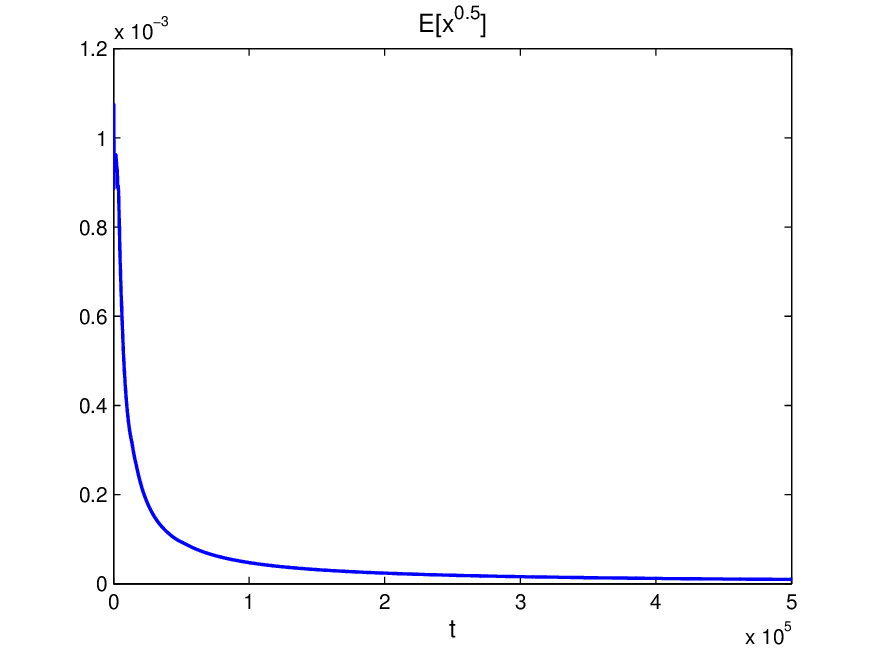}}\\
\footnotesize {\textbf{Figure 7.}\;\;
The 0.5th moment of $x(t)$ along with Markov chain $r_t$ is exponential stable, where $x(t)$ with initial value $x(0)=1$ and computed by step $\Delta=1e-3$.}
\end{center}

\begin{rem}
In this section, by using numerical method we verify all the theoretical results obtained in this paper. Meanwhile, from these two examples one can see the special character of regime switching system, that is, the stability of some subsystems does not implies the stability of the whole system, which is shown by Example 1, and Example 2 shows that the unstability of some subsystems does not yields the  unstability of the whole system.
\end{rem}

\section {Conclusion}
In this paper, we study  a stochastic GA model perturbed by regime switching and white noise.  Especially, we allow the GA parameters $\theta$  and $\gamma$ to   
vary according to a continuous-time Markov chain, reflecting the fact that the GA parameters may change in different environments. Compared with the models in the literatures(e.g., \cite{ga-73,li-13,liu-11,liu-12-a,liu-12-b,liu-15,Sa-15}), our formulation  provides a more realistic modeling of the population  dynamics, which also includes some mean
reverting  models in financial fields, such as the 3/2 model, Logistic diffusion model and Double-Well potential model. However, our model does introduce extra difficulty in the analysis because of the regime switching mechanism. We overcome these difficulties  by constructing  suitable Lyapunov functions and using some analysis technics.
 We show the existence of global positive solutions,  asymptotic stability in probability and extinction of model \eqref{ga} under the condition  $\sum_{i\in \S}\pi_i\mu(i)<0$, and obtain the weak persistence of the species described by the model under the condition that $\sum_{i\in \S}\pi_i\mu(i)>0$, also we get that the amplitude of the solution $x(t)$ to the model is at least $[x_*, x^*]$, which generalize the previous results. Meanwhile, we prove that its trivial solution is $p$th moment exponential stable by using the properties of nonsingular $M$-matrix $\A(p)$, but from the computer simulation, we know that this condition is not necessary. Under the assumptions  $\theta(i)+1\gt2\gamma(i)>0$ for each $i\in\S$ and $\sum_{i\in\S}\pi_i\mu(i)>0,$ we prove that, for any $(x_0,r_0)\in\R_+\times\S,$ the  solution $x(t)$ of model \eqref{ga} is positive recurrent and admits a unique ergodic  asymptotically invariant distribution $\nu$. Moreover, the species $x(t)$ described by model \eqref{ga} is stochastically permanent. Even for the case of no regime switching in the GA parameters, our condition imposed on  GA parameters  $\theta$ and $\gamma$ is $\theta+1\gt2\gamma>0$, which improves those in the previous works.

\section {Acknowledgments}
The authors  are very grateful to the referees and editors  for providing us with detailed comments and
suggestions for improving the quality of this paper.
The work was supported  by the NSFC (61703001),  the NSF
of Anhui Province (1708085MA17, 1508085QA13), the Key NSF of Education Bureau of Anhui Province
(KJ2018A0437) and the Support Plan of Excellent Youth Talents in Colleges and Universities in Anhui Province (gxyq2017011).

 %
\vspace{3mm}


{\footnotesize
\begin{thebibliography}{99}
\setlength{\itemsep}{0pt}
\setlength{\parskip}{0pt}\baselineskip=6.mm
\bibitem{ad-09} Applebaum D., L$\acute{e}$vy Processes and Stochastic Calculus.  Cambridge University Press, 2009.

\bibitem{ga-73} Gilpin, M.E., Ayala, F.J., Global models of growth and competition. Proc. Natl.
Acad. Sci. 70 (1973) 3590-3593.

\bibitem{kh-80} Khasminskii R.Z., Stochastic Stability of Differential Equations. Springer-Verlag Berlin Heidelberg, 2012.

\bibitem{kh-07} Khasminskii R.Z., Zhu C., Yin G., Stability of regime-switching diffusions.
Stoch. Process  Appl. 117 (2007) 1037-1051.

\bibitem{Higham-01} Higham  D.J., An Algorithmic Introduction to Numerical Simulation of
Stochastic Differential Equations. SIAM Review 43 (2001) 525-546.

\bibitem{Lewis} Lewis A.L., Option Valuation Under Stochastic Volatility, Finance Press, 2000.

\bibitem{li-13} Li D., The stationary distribution and ergodicity of a
stochastic generalized logistic system. Stat. Prob. Lett. 83 (2013) 580-583.

\bibitem{liu-15} Liu M.,  Li Y., Stability of a stochastic logistic model
under regime switching. Adva. Diff. Equat. 2015 (2015) 326-334.

\bibitem{liu-11} Liu M., Wang K., Asymptotic properties and simulations of a stochastic logistic model
under regime switching. Math. Computer Modelling 54 (2011) 2139-2154.

\bibitem{liu-12-a} Liu M., Wang K.,  Asymptotic properties and simulations of a stochastic logistic model
under regime switching II. Math. Computer Modelling  55 (2012) 405-418.

\bibitem{liu-12-b} Liu M., Wang K.,  Stationary distribution, ergodicity and extinction of a
stochastic generalized logistic system. Appl. Math. Lett. 25 (2012) 1980-1985.


\bibitem{mao-02} Mao X., Marion G., Renshaw, E.,  Environmental Brownian noise suppresses
explosions in population dynamics. Stoch. Process. Appl. 97 (2002) 95-110.

\bibitem{mao-06} Mao X., Yuan C., Stochastic differential equations with Markovian switching.  Imperial College Press, 2006.

\bibitem{may-73} May R.M., Stability and Complexity in Model Ecosystems. Princeton University
Press, 1973.

\bibitem{Sa-15}
Settati A., Lahrouz A.,  On stochastic Gilpin-Ayala population model with markovian switching. BioSystems 130  (2015) 17-27.


\bibitem{zhu-07} Zhu C., Yin G.,  Asymptotic properties of hybrid diffusion systems. SIAM J.
Control Optim. 46 (2007) 1155-1179.

\bibitem{zhu-09} Zhu C., Yin G.,  On competitive Lotka-Volterra model in random environments.
J. Math. Anal. Appl. 357  (2009)  154-170.



\end{thebibliography}}
\end{document}